\documentclass[12pt]{amsart}
\usepackage{graphicx}
\usepackage[headings]{fullpage}
\usepackage{bm,amssymb,epic,eepic,epsfig,amsbsy,amsmath,amscd}
\numberwithin{equation}{section}
                        \theoremstyle{plain}
\usepackage{mathrsfs}

\newcommand\no[1]{}

\newtheorem{theorem}{Theorem}[section]
\newtheorem{thm}{Theorem}
\newtheorem{lemma}[theorem]{Lemma}

\theoremstyle{definition}
\newtheorem{remark}[theorem]{Remark}

\newcommand{\la}{\langle}
\newcommand{\ra}{\rangle}

\def\BZ{\mathbb Z}
\def\BR{\mathbb R}

\def\la{\langle}
\def\ra{\rangle}

\def\be { \begin{equation} }
\def\ee { \end{equation} }

\begin{document}

\title[LO surgeries on twisted torus knots]
{Left-orderability for surgeries on twisted torus knots}

\author{Anh T. Tran}

\begin{abstract}
We show that the fundamental group of the $3$-manifold obtained by $\frac{p}{q}$-surgery along the $(n-2)$-twisted $(3,3m+2)$-torus knot, with $n,m \ge 1$, is not left-orderable if $\frac{p}{q} \ge 2n + 6m-3$ and is left-orderable if $\frac{p}{q}$ is sufficiently close to $0$.
\end{abstract}



\address{Department of Mathematical Sciences, The University of Texas at Dallas, Richardson, TX 75080, USA}
\email{att140830@utdallas.edu}

\maketitle

\section{Introduction}

The motivation of this paper is the L-space conjecture of Boyer, Gordon and Watson \cite{BGW} which states that an irreducible rational homology 3-sphere is an L-space if and only if its fundamental group is not left-orderable. Here a rational homology 3-sphere $Y$ is an $L$-space if its Heegaard Floer homology $\widehat{HY}(Y)$ has rank equal to the order of $H_1(Y; \BZ)$, and a nontrivial group $G$ is left-orderable if it admits a total ordering $<$ such that $g<h$ implies $fg < fh$ for all elements $f,g,h$ in $G$. 

Many hyperbolic $L$-spaces can be obtained via Dehn surgery. A knot $K$ in $S^3$ is called an L-space knot if it admits a positive Dehn surgery yielding an L-space. For an L-space knot $K$, Ozsvath and Szabo \cite{OS} proved that the $\frac{p}{q}$-surgery of $K$ is an $L$-space if and only if $\frac{p}{q}\ge 2g(K)-1$, where $g(K)$ is the genus of $K$. In view of the L-space conjecture, one would expect that the fundamental group of the $\frac{p}{q}$-surgery of an L-space knot $K$ is not left-orderable if and only if $\frac{p}{q} \ge 2g(K)-1$. 

By \cite{BM} among all Montesinos knots, the $(-2,3,2n+1)$-pretzel knots, with $n \ge 3$, and their mirror images are the only hyperbolic L-space knots. Nie \cite{Nie} has recently proved that the fundamental group of the $3$-manifold obtained by $\frac{p}{q}$-surgery along the $(-2,3,2n+1)$-pretzel knot, with $n \ge 3$, is not left-orderable if $\frac{p}{q} \ge 2n + 3$ and is left-orderable if $\frac{p}{q}$ is sufficiently close to $0$. This result extends previous ones by Jun \cite{Jun}, Nakae \cite{Nakae}, and Clay and Watson \cite{CW}. Note that the genus of the $(-2,3,2n+1)$-pretzel knot, with $n \ge 3$, is equal to $n+2$. 
 
 In this paper, we study the left-orderability for surgeries on the twisted torus knots. Some results about non left-orderable surgeries of  twisted torus knots were obtained by Clay and Watson \cite{CW}, Ichihara and Temma \cite{IT2015, IT2018}, and Christianson, Goluboff, Hamann, and Varadaraj \cite{C}. We will focus our study to the $(n-2)$-twisted $(3,3m+2)$-torus knots, which are the knots obtained from the $(3,3m+2)$-torus knot by adding $(n-2)$ full twists along an adjacent pair of strands. For $n,m \ge 1$,  these knots are known to be L-space knots, see \cite{Va}. Moreover, the $(n-2)$-twisted $(3,5)$-torus knots are exactly the $(-2,3,2n+1)$-pretzel knots. Note that the genus of the $(n-2)$-twisted $(3,3m+2)$-torus knot, with $n,m \ge 1$, is equal to $n+3m-1$. 
 
 The following result generalizes the one in \cite{Nie}. 
 
\begin{thm} \label{main}
Suppose $n, m \ge 1$. Then the fundamental group of the $3$-manifold obtained by $\frac{p}{q}$-surgery along the $(n-2)$-twisted $(3,3m+2)$-torus knot is 

(i) not left-orderable if $\frac{p}{q} \ge 2n + 6m-3$, 

(ii) left-orderable if $\frac{p}{q}$ is sufficiently close to $0$.\end{thm}

The rest of this paper is devoted to the proof of Theorem \ref{main}. In Section \ref{nonLO} we prove part (i). To do so, we follow the method of Jun \cite{Jun}, Nakae \cite{Nakae} and Nie \cite{Nie} which was developed for studying the non left-orderable surgeries of the $(-2,3,2n+1)$-pretzel knots. In Section \ref{LO} we prove part (ii). To this end, we apply a criterion for the existence of left-orderable surgeries of knots which was first developed by Culler and Dunfield \cite{CD}, and then improved by Herald and Zhang \cite{HZ}.

\section{Non left-orderable surgeries} \label{nonLO}

Let $K_{n,m}$ denote the $(n-2)$-twisted $(3,3m+2)$-torus knot. By \cite{IT2018} (see also \cite{IT2015}, \cite{CW}), the knot group of $K_{n,m}$ has a presentation 
$$\pi_1(S^3 \setminus K_{n,m}) = \la a, w \mid w^n (aw)^m a^{-1} (aw)^{-m} = (wa)^{-m} a (wa)^m w^{n-1} \ra,$$ where $a$ is a meridian. Moreover, the preferred longitude corresponding to $\mu = a$ is 
\begin{equation} \label{long}
\lambda = a^{-(4n+9m-2)} [(wa)^m w^{n}] (aw)^{m-1} a [w^{n} (aw)^m].
\end{equation}
Note that the first homology class of $w$ is twice that of the meridian $a$.

\begin{remark}
(i) It is known that $K_{n,1}$ is the $(-2,3,2n+1)$-pretzel knot. The above presentation of the knot group of $K_{n,1}$ was first derived in \cite{LT}, \cite{Nakae}. 

(ii) The formula \eqref{long} for the longtitude of $K_{n,m}$ in \cite{IT2015}, \cite{IT2018} contains a small error: $a^{-(4n+9m-2)}$ was written as $a^{-(2n+9m+2)}$.
\end{remark}

Let $M_{\frac{p}{q}}$ be the 3-manifold obtained by $\frac{p}{q}$-surgery along the $(n-2)$-twisted $(3,3m+2)$-torus knot $K_{n,m}$. Then 
$$\pi_1(M_{\frac{p}{q}}) = \la a, w \mid w^n (aw)^m a^{-1} (aw)^{-m} = (wa)^{-m} a (wa)^m w^{n-1}, a^p \lambda^q =1 \ra.$$

Since $a^p \lambda^q =1$ in $\pi_1(M)$ and $a\lambda = \lambda a$, there exists an element $k \in \pi_1(M)$ such that $a = k^q$ and $\lambda = k^{-p}$, see e.g. \cite[Lemma 3.1]{Nakae}. 

Suppose $m, n \ge 1$. Assume $\pi_1(M_{\frac{p}{q}})$ is left-orderable for some $\frac{p}{q} \ge 2n + 6m-3$, where $q>0$. Then there exists a monomorphism $\rho: \pi_1(M_\frac{p}{q}) \to \text{Homeo}^+(\BR)$ such that there is no $x \in \BR$ satisfying $\rho(g)(x) = x$ for all $g \in \pi_1(M)$, see e.g. \cite[Problem 2.25]{CR}. 

From now on we write $gx$ for $\rho(g)(x)$.

\begin{lemma}
We have $kx \not=x$ for any $x \in \BR$.
\end{lemma}

\begin{proof}
Assume $kx=x$ for some $x \in \BR$. Then $x=k^q x=ax$. If $x=wx$ then $gx=x$ for all $g \in \pi_1(M)$, a contradiction. Otherwise, without loss of generality, we assume that $x < wx$. Then we have
\begin{eqnarray*}
x &=& a^{(4n+9m-2)} k^{-p} x = a^{(4n+9m-2)} \lambda x =  [(wa)^m w^{n}] (aw)^{m-1} a [w^{n} (aw)^m]x > x,
\end{eqnarray*}
which is also a contradiction.
\end{proof}

Since $kx \not=x$ for any $x \in \BR$ and $kx$ is a continuous function of $x$, without loss of generality, we may assume $x < kx$ for any $x \in \BR$. Then $x < k^q x = ax$. 

\begin{lemma} \label{base}
We have $(aw)^m a x < w(aw)^m x$ for any $x \in \BR$. \end{lemma}

\begin{proof}
Since $w^n (aw)^m a^{-1} (aw)^{-m} = (wa)^{-m} a (wa)^m w^{n-1}$ in $\pi_1(M_{\frac{p}{q}})$, we have
\begin{eqnarray*}
w(aw)^m x  &=& [(aw)^m a (aw)^{-m} w^{-n} (wa)^{-m} a (wa)^m w^{n-1}] w(aw)^m x \\
                   &=& (aw)^m a  [(wa)^m w^{n}(aw)^m]^{-1} a [(wa)^m w^{n}(aw)^m] x.
\end{eqnarray*}

Writing $g$ for $(wa)^m w^{n}(aw)^m$, we then obtain
$$w(aw)^m x  = (aw)^m a g^{-1} a g x > (aw)^m a x,$$
since $g^{-1} a g x > g^{-1} g x = x$. 
\end{proof}

Lemma \ref{base} implies that $(aw)^m  x < (aw)^m a x < w(aw)^m x$. Hence $x < wx$ for any $x \in \BR$.

\begin{lemma} \label{induction}
For any $x \in \BR$ and $k \ge 1$ we have 
$$(aw)^m a^k x < w^k(aw)^m x \qquad \text{and} \qquad a^k (wa)^m x < (wa)^m w^k x.$$
\end{lemma}

\begin{proof}
We prove the lemma by induction on $k \ge 1$. The base case ($k=1$) is Lemma \ref{base}. Assume $(aw)^m a^k x < w^k(aw)^m x$ for any $x \in \BR$. Then
\begin{eqnarray*}
(aw)^m a^{k+1} x &=& (aw)^m a^{k} (ax) \\
                            &<& w^k(aw)^m ax \\
                            &<& w^k (wa)^m w x = w^{k+1}(aw)^m x.
\end{eqnarray*}
Similarly, assuming $a^k (wa)^m x < (wa)^m w^k x$ for any $x \in \BR$ then
\begin{eqnarray*}
a^{k+1} (wa)^m x < a (wa)^m w^k x = (aw)^m a w^k x< w (aw)^m w^{k} x = (wa)^m w^{k+1} x.
\end{eqnarray*}
This completes the proof of Lemma \ref{induction}.
\end{proof}

\begin{lemma} \label{key}
With $\frac{p}{q} \ge 2n + 6m-3$ we have $wx< ax$ for any $x \in \BR$.
\end{lemma}

\begin{proof}
With $\frac{p}{q} \ge 2n + 6m-3$ and $q>0$, we have $-p + (2n + 6m-3)q \le  0$. Since $a = k^q$, $\lambda = k^{-p}$ and $x < kx$ for any $x \in \BR$, we have 
\begin{eqnarray*}
ax &\ge& k^{-p+(2n + 6m-3)q} ax = a^{2n + 6m-2}\lambda x \\
&=& a^{-n} [(wa)^m w^{n}] (aw)^{m-1}a [w^{n} (aw)^m]  a^{-(n+3m)} x.
\end{eqnarray*}

Then, by Lemma \ref{induction}, we obtain
\begin{eqnarray*}
ax 
&>& a^{-n} [a^{n} 
(wa)^{m}]  (aw)^{m-1}a  [(aw)^m  a^n] a^{-(n+3m)} x 
\\
&=& 
w(aw)^{m-1} a (aw)^{m-1} a  (aw)^m a^{-3m} x \\
&>& w a^{m-1} a a^{m-1}a  a^{m}   a^{-3m} x = wx.
\end{eqnarray*}
Here, in the last inequality, we use the fact that $x < wx$ for any $x \in \BR$.
\end{proof}

With $\frac{p}{q} \ge 2n + 6m-3$, by Lemmas  \ref{induction} and \ref{key} we have \begin{eqnarray*}
(aw)^m x &=& [(aw)^m a] a^{-1} x \\
                &<& [w(aw)^m ] a^{-1} x = (wa)^m w (a^{-1} x) \\
                &<& (wa)^m a (a^{-1} x) = a^{-1} [(aw)^m a]  x \\
                &<& a^{-1} [w(aw)^m ] x = a^{-1} w[(aw)^m x]\\
                &<& a^{-1} a[(aw)^m x] = (aw)^m x,
\end{eqnarray*}
a contradiction. This proves Theorem \ref{main}(i).

\section{Left-orderable surgeries} \label{LO}

To prove Theorem \ref{main}(ii) we apply the following result. It was first stated and proved by Culler and Dunfield \cite{CD} under an additional condition on $K$.

\begin{theorem} \label{HZ}
\cite{HZ} 
For a knot $K$ in $S^3$, if its Alexander polynomial $\Delta_K(t)$ has a simple root on the unit circle, then the fundamental group of the
manifold obtained by $\frac{p}{q}$-surgery along $K$ is left-orderable if $\frac{p}{q}$ is sufficiently close to $0$.
\end{theorem}

In view of Theorem \ref{HZ}, to prove Theorem \ref{main}(ii) it suffices to show that the Alexander polynomial of the twisted torus knot $K_{n,m}$ has a simple root on the unit circle. The rest of the paper is devoted to the proof of this fact. We start with a formula for the Alexander polynomial of a knot via Fox's free calculus. 

\subsection{The Alexander polynomial} 

Let $K$ be a knot in $S^3$ and $E_K = S^3 \setminus K$ its complement. We choose a deficiency one presentation for the knot group of $K$:
$$
\pi_1(E_K)=
\langle a_1,\ldots,a_l~|~r_1,\ldots,r_{l-1}\rangle.
$$
Note that this does not need to be a Wirtinger presentation. 
Consider the abelianization homomorphism 
$\alpha:\pi_1(E_K)\to H_1(E_K;\BZ)
\cong {\BZ}
=\langle t
\rangle$. 

The map $\alpha$ naturally induces a ring homomorphism $\tilde{\alpha}:{\BZ}[\pi_1(E_K)]\rightarrow {\BZ}[t^{\pm1}]$, 
where ${\BZ}[\pi_1(E_K)]$ is the group ring of $\pi_1(E_K)$. Consider the $(l-1)\times l$ matrix $A$ 
whose $(i,j)$-entry is $\tilde{\alpha}(\frac{\partial r_i}{\partial a_j})
\in {\BZ}[t^{\pm1}]$,
where 
$\frac{\partial}{\partial a}$ 
denotes the Fox's free differential. 
For 
$1\leq j\leq l$, 
denote by $A_j$ 
the $(l-1)\times(l-1)$ matrix obtained from $A$ 
by removing the $j$th column. 
Then it is known that the rational function 
$$\frac{\det A_j}{\det\tilde{\alpha}(a_j-1)}$$ is an invariant of $K$, see e.g. \cite{Wa}. 
It is well-defined 
up to a factor $\pm t^{k}~(k\in{\BZ})$ and is related to the Alexander polynomial $\Delta_K(t)$ of $K$ by the following formula
$$\frac{\det A_j}{\det\tilde{\alpha}(a_j-1)} = \frac{\Delta_K(t)}{t-1}.$$

\subsection{Proof of Theorem \ref{main}(2)}

Let $r_1 = w^n (aw)^m a^{-1}(aw)^{-m}$ and $r_2 = (wa)^{-m} a (wa)^m w^{n-1}$. Then we can write $\pi_1(E_{K_{n,m}}) = \la a, w \mid r_1 r_2^{-1} = 1 \ra$. In $\pi_1(E_{K_{n,m}})$ we have 
$$
\frac{\partial r_1 r_2^{-1}}{\partial a} = \frac{\partial r_1}{\partial a} + r_1 \frac{\partial  r_2^{-1}}{\partial a} =  \frac{\partial r_1}{\partial a} - r_1r_2^{-1} \frac{\partial r_2}{\partial a} = \frac{\partial r_1}{\partial a} -\frac{\partial r_2}{\partial a}.
$$

Let $\delta_k(g) = 1 + g + \cdots + g^k$. Then 
\begin{eqnarray*}
\frac{\partial r_1 r_2^{-1}}{\partial a} 
&=& w^n \big[ \delta_{m-1}(aw) - (aw)^m a^{-1}(aw)^{-m} \big( \delta_{m-1}(aw) + (aw)^m \big) \big]  \\
&& - \,  \big[ -(wa)^{-m} \delta_{m-1}(wa) w + (wa)^{-m} \big( 1 + a \delta_{m-1}(wa) w \big) \big] \\
&=& - w^n (aw)^m a^{-1} \big[ 1 - (a-1) (aw)^{-m}  \delta_{m-1}(aw) \big]  \\
&& - \,  (wa)^{-m}  \big[ 1 + (a-1)w \delta_{m-1}(aw)   \big].
\end{eqnarray*}

The Alexander polynomial $\Delta_{K_{n,m}}(t)$ of $K_{n,m}$ satisfies $$\frac{\Delta_{K_{n,m}}(t)}{t-1} =  \frac{\tilde{\alpha}(\frac{\partial r_1 r_2^{-1}}{\partial a})}{\tilde{\alpha}(w)-1}.$$ Hence, since $\tilde{\alpha}(a)=t$ and $\tilde{\alpha}(w) = t^2$, we have
\begin{eqnarray*}
-(t+1)\Delta_{K_{n,m}}(t) &=& t^{2n+3m-1} \big[ 1 - (t-1) t^{-3m} \delta_{m-1}(t^3)\big] + t^{-3m} \big[ 1+(t-1)t^2   \delta_{m-1}(t^3)\big] \\
&=& t^{2n+3m-1} + t^{-3m} - (t^{2n-1} - t^{2-3m}) (t-1)\delta_{m-1}(t^3) \\
&=& t^{2n+3m-1} + t^{-3m} - (t^{2n-1} - t^{2-3m}) (t^{3m}-1)/(t^2+t+1) \\
&=& t^{-3m}(1+t+t^{3m+2}+t^{2n+3m-1}+t^{2n+6m}+t^{2n+6m+1})/(t^2+t+1).
\end{eqnarray*}

Let $f(t) = t^{n+3m+1/2} + t^{-(n+3m+1/2)}+t^{n+3m-1/2} + t^{-(n+3m-1/2)} + t^{n-3/2} + t^{-(n-3/2)}$. Then 
$$
\Delta_{K_{n,m}}(t) = - \frac{t^{n-1}f(t)}{(t^{1/2}+t^{-1/2})(t+t^{-1}+1)}.
$$
Hence $
\Delta_{K_{n,m}}(e^{i\theta}) = - \frac{e^{i(n-1)\theta}f(e^{i\theta})}{2\cos(\theta/2)(2\cos\theta+1)}.
$

Let $g(\theta) = f(e^{i\theta})/2$. To show that $\Delta_{K_{n,m}}(t)$ has a simple root on the unit circle, it suffices to show that $g(\theta)$ has a simple root on $(0, 2\pi/3)$. We have
\begin{eqnarray*} 
g(\theta) &=& \cos (n+3m+1/2)\theta + \cos (n+3m-1/2)\theta + \cos (n-3/2) \theta \\
              &=& 2 \cos (\theta/2) \cos (n+3m)\theta + \cos (n-3/2) \theta.
\end{eqnarray*}

If $n=1$ then $g(\theta) = \cos (\theta/2) (2\cos (n+3m)\theta + 1)$. It is clear that $\theta = \frac{2\pi/3}{n+3m}$ is a simple root of $g(\theta)$ on $(0,\pi/6]$. 

Suppose $n \ge 2$. We claim that $g(\theta)$ has a simple root on $(\theta_0, \theta_1)$ where $\theta_0 = \frac{\pi/2}{n+3m}$ and $\theta_1 = \frac{\pi/2}{n+3m/2-3/4}$. 
Note that $0 < \theta_0 < \theta_1 \le \frac{\pi/2}{7/4}= \frac{2\pi}{7}$. We have
$$
g(\theta_0) = \cos (n-3/2) \theta_0 = \cos \left(\frac{\pi}{2} \cdot \frac{n-3/2}{n+3m} \right) > 0,
$$
since $0 < \frac{\pi}{2} \cdot \frac{n-3/2}{n+3m} < \frac{\pi}{2}$. At $\theta = \theta_1 = \frac{\pi/2}{n+3m/2-3/4}$ we have 
$$\cos (n+3m)\theta + \cos (n-3/2) \theta = 2 \cos (n+3m/2-3/4)\theta \cos (3m/2+3/4)\theta=0.$$ Hence 
\begin{eqnarray*} 
g(\theta_1) &=& (1-2\cos(\theta_1/2))  \cos (n-3/2) \theta_1 \\
&=& (1-2\cos(\theta_1/2))  \cos \left(\frac{\pi}{2} \cdot \frac{n-3/2}{n+3m/2-3/4} \right) < 0,
\end{eqnarray*}
since $1-2\cos(\theta_1/2) < 0 < \cos \left(\frac{\pi}{2} \cdot \frac{n-3/2}{n+3m/2-3/4} \right)$. 

We show that $g(\theta)$ is a strictly decreasing function on $(\theta_0, \theta_1)$. Indeed, we have 
\begin{eqnarray*} 
-g'(\theta) &=& \sin (\theta/2) \cos (n+3m)\theta + 2 (n+3m)\cos (\theta/2) \sin (n+3m)\theta \\
  && + \, (n-3/2) \sin (n-3/2) \theta
\end{eqnarray*}

Since $$0 < (n-3/2) \theta < \frac{\pi}{2} \cdot \frac{n-3/2}{n+3m/2-3/4} < \frac{\pi}{2},$$ we have $(n-3/2) \sin (n-3/2) \theta >0$. Since $\frac{\pi/2}{n+3m} < \theta <\frac{\pi/2}{n+3m/2-3/4}$ we have 
$${\pi/2} < (n+3m)\theta <\frac{\pi}{2} \cdot \frac{n+3m}{n+3m/2-3/4} < \pi,$$
which implies that $\cos (n+3m)\theta < 0 <\sin (n+3m)\theta$. Hence
$$
-g'(\theta) > \sin (\theta/2) \cos (n+3m)\theta + \cos (\theta/2) \sin (n+3m)\theta = \sin (n+3m+1/2)\theta.
$$
Since $0< (n+3m+1/2)\theta < \frac{\pi}{2} \cdot \frac{n+3m+1/2}{n+3m/2-3/4} \le \pi$, we have $\sin (n+3m+1/2)\theta \ge 0$. Hence $-g'(\theta) > 0$ on $(\theta_0, \theta_1)$. This, together with $g(\theta_0) > 0 > g(\theta_1)$, implies that $g(\theta)$ has a simple root on $(\theta_0, \theta_1)$. The proof of Theorem \ref{main}(ii) is complete.

\section*{Acknowledgements} 
The author has been partially supported by a grant from the Simons Foundation (\#354595 to AT).
He would like to thank K. Ichihara for helpful discussions.


\begin{thebibliography}{99999}

\baselineskip15pt

\bibitem[BGW]{BGW} S. Boyer, C. Gordon and L. Watson, {\em On L-spaces and left-orderable fundamental groups},  Math. Ann. \textbf{356} (2013), no. 4, 1213--1245. 

\bibitem[BM]{BM} K. Baker and A. Moore, {\em Montesinos knots, Hopf plumbings, and L-space surgeries}, J. Math. Soc. Japan \textbf{70} (2018), no. 1, 95--110.

\bibitem[CD]{CD} N. Dunfield and M. Culler, {\em Orderability and Dehn filling}, Geom. Topol. \textbf{22} (2018), 1405--1457.

\bibitem[CGHV]{C} K. Christianson, J. Goluboff, L. Hamann, and S. Varadaraj, {\em Non-left-orderable surgeries on twisted torus knots}, Proc. Amer. Math. Soc. \textbf{144} (2016), no. 6, 2683--2696.

\bibitem[CR]{CR} A. Clay and D. Rolfsen, Ordered groups and topology, Graduate Studies in Mathematics, 176. American Mathematical Society, Providence, RI, 2016. x+154 pp.

\bibitem[CW]{CW} A. Clay and L. Watson, {\em Left-orderable fundamental groups and Dehn surgery}, Int. Math. Res. Not. IMRN 2013, no. 12, 2862--2890.

\bibitem[HZ]{HZ} C. Herald and X. Zhang, {\em A note on orderability and Dehn filling},  arXiv:1807.00742. 

\bibitem[IT1]{IT2015} K. Ichihara and Y. Temma, {\em Non-left-orderable surgeries and generalized Baumslag-Solitar relators}, J. Knot Theory Ramifications \textbf{24} (2015), no. 1, 1550003, 8 pp.

\bibitem[IT2]{IT2018} K. Ichihara and Y. Temma, {\em Non-left-orderable surgeries on negatively twisted torus knots}, Proc. Japan Acad. Ser. A Math. Sci. \textbf{94} (2018), no. 5, 49--52. 

\bibitem[Ju]{Jun} J. Jun, {\em $(-2,3,7)$-pretzel knot and Reebless foliation}, Topol. Appl. \textbf{145}(1-3) (2004) 209--232.

\bibitem[Na]{Nakae} Y. Nakae, {\em A good presentation of $(-2,3,2s+1)$-type pretzel knot group and $\BR$-covered foliation}, J. Knot Theory Ramifications \textbf{22} (2013), no. 1, 1250143, 23 pp.

\bibitem[Ni]{Nie} Z. Nie, {\em Left-orderablity for surgeries on $(-2,3,2s+1)$-pretzel knots}, arXiv:1803.00076. 

\bibitem[LT]{LT} T. Le and A. Tran, {\em On the AJ conjecture for knots}, with an appendix written jointly with V. Huynh, Indiana Univ. Math. J. \textbf{64} (2015), no. 4, 1103--1151.

\bibitem[OS]{OS} P. Ozsvath and Z. Szabo, {\em Knot Floer homology and rational surgeries}. Algebr. Geom. Topol. \textbf{11} (2011), no. 1, 1--68.

\bibitem[Va]{Va} F. Vafaee, {\em On the knot Floer homology of twisted torus knots}, Int. Math. Res. Not. IMRN 2015, no. 15, 6516--6537.

\bibitem[Wa]{Wa}
M. Wada,
\textit{Twisted Alexander polynomial for finitely
presentable groups},
Topology {\bf 33} (1994), 241--256.

\end{thebibliography}
\end{document}